\documentclass[11pt]{article}
\usepackage{latexsym,amsmath,amssymb,amsfonts,mathrsfs,amsthm}
\usepackage{epsf,graphicx,epsfig,color,cite,cases}
\usepackage{subfigure,graphics,multirow,marginnote,enumerate,bm}
\usepackage[T1]{fontenc}
\usepackage{algorithm}

\newcommand{\R}{{\mat R}}

\newcommand{\N}{{\mat N}}

\newcommand{\ds}{\displaystyle}
\newcommand{\no}{\nonumber}
\newcommand{\be}{\begin{eqnarray}}
\newcommand{\ben}{\begin{eqnarray*}}
\newcommand{\en}{\end{eqnarray}}
\newcommand{\enn}{\end{eqnarray*}}

\newcommand{\pa}{\partial}

\newcommand{\ov}{\overline}

\newcommand{\G}{\Gamma}

\newcommand{\Om}{\Omega}

\newcommand{\mat}{\mathbb}
\newcommand{\se}{\setminus}

\newtheorem{theorem}{Theorem}[section]

\begin{document}
\renewcommand{\theequation}{\arabic{section}.\arabic{equation}}

\title{\bf Simultaneous recovery of a locally rough interface and the embedded obstacle with its surrounding medium}

\author{Jiaqing Yang\thanks{School of Mathematics and Statistics, Xi'an Jiaotong University,
Xi'an, Shaanxi 710049, China ({\tt jiaq.yang@mail.xjtu.edu.cn})}
\and
Jianliang Li\thanks{School of Mathematics and Statistics, Changsha University of Science
and Technology, Changsha 410114, China ({\tt lijl@amss.ac.cn})}
\and
Bo Zhang\thanks{LSEC, NCMIS and Academy of Mathematics and Systems Science, Chinese Academy
of Sciences, Beijing 100190, China and School of Mathematical Sciences, University of Chinese
Academy of Sciences, Beijing 100049, China ({\tt b.zhang@amt.ac.cn})}
}
\date{}

\maketitle

\begin{abstract}
Consider the scattering of time-harmonic point sources by an infinite locally rough interface with bounded obstacles embedded in the lower half-space. The model problem is first reduced to an equivalent 
integral equation formulation defined in a bounded domain, where the well-posedness is obtained in $L^p$ by the classical Fredholm theory. Then a global uniqueness theorem is proved for the inverse problem
 of recovering the locally rough interface, the embedded obstacles and the wave number in the lower-half space by means of near-field measurements above the interface. 
\vspace{.2in}


\end{abstract}

\maketitle

\section{Introduction}
This paper is concerned with the two-dimensional inverse scattering of time-harmonic acoustic point sources by a locally rough interface with obstacles embedded in the lower half-space. This type of problems can find applications in diverse scientific areas such as radar, underwater exploration and non-destructive testing, where the shape, location and boundary conditions of both the interface and embedded obstacles need to be simultaneously
reconstructed by the measurements of the scattered fields taken on some certain subdomain of the upper-half space. 

\begin{figure}[htbp]
\centering
\includegraphics[width=5in, height=2.5in]{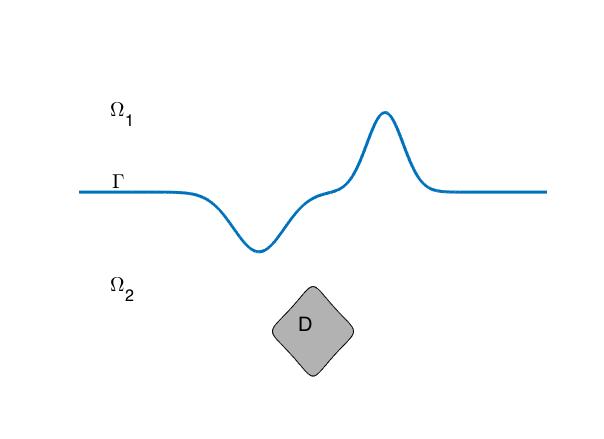}
\caption{The physical configuration of the scattering problem.}
\label{f1} 
\end{figure}

Let the scattering interface be denoted by a smooth curve $\Gamma:=\{(x_1,x_2)\in\R^2: x_2=f(x_1)\}$, where $f$ is assumed to be a Lipschitz continuous function with compact support.
This means that $\Gamma$ is just a local perturbation of the planar interface $\Gamma_0:=\{(x_1,x_2)\in\R^2: x_2=0\}$. The whole space $\R^2$ is then separated by $\G$ into 
the upper half-space $\Omega_1:=\{(x_1,x_2)\in\R^2: x_2>f(x_1)\}$ and the lower half-space $\Omega_2:=\{(x_1,x_2)\in\R^2: x_2<f(x_1)\}$, where a bounded obstacle $D$ is embedded into 
$\Om_2$ which is assumed to be of $C^{2}$-class. We refer the reader to Figure \ref{f1} for a geometrical configuration for the problem setting.

Consider the incident field to be generated by a point source
\ben
u^{\rm inc}(x, x_s) = \Phi_{\kappa_1}(x, x_s):=\frac{\rm i}{4}H_0^{(1)}(\kappa_1|x-x_s|),\qquad x_s\in\Om_1
\enn
which corresponds to the fundamental solution of the Helmholtz equation $\Delta \Phi_{\kappa_1}(x, x_s)+\kappa_1^2\Phi_{\kappa_1}(x, x_s)=-\delta(x-x_s)$ in $\R^2$. Here, $\delta$ is the Kronecker delta distribution.
Then the scattering of $u^{\rm inc}(x, x_s)$ by the scatterers $(\Gamma, D)$ can be modeled by 
\be\label{a1}
\left\{\begin{array}{lll}
         \Delta u+\kappa^2u=0  &\textrm{in}\;\; \R^2\se(\ov{D}\cup\{x_s\}), \\[2mm]
         \mathcal{B}u=0&\textrm{on}\;\; \pa{D},\\
         \ds\lim_{r\rightarrow \infty}\sqrt{r}\left(\frac{\partial u^{s}}{\partial r}-{\rm i}\kappa u^{s}\right)=0& {\rm for}\;\;r=|x|
       \end{array}
\right.
\en
if $D$ is an impenetrable obstacle, where $u$ denotes the total field consisting of the point source $u^{\rm inc}(x, x_s)$ and the scattered field $u^s(x, x_s)$ in $\Omega_1$, and $u:=u^{s}(x, x_s)$ denotes the transmitted field in $\Omega_2\setminus\overline{D}$. Here, $\kappa>0$ is the wave number satisfying 
$\kappa:=\kappa_1\in\R$ in $\Omega_1$ and $\kappa:=\kappa_2\in\R$ in $\Omega_2$. Moreover, $\mathcal{B}$ stands for the boundary condition on the boundary $\partial D$ satisfying $\mathcal{B}u:=u$ if $D$ is 
sound-soft, and $\mathcal{B}u:=\partial_{\nu} u+{\rm i}\lambda u$ for a continuous impedance function $\lambda(x)\geq0$ if $D$ is an imperfect obstacle. Here, $\nu$ is the outward normal vector directing into 
$\Omega_2\se\ov{D}$. If $\lambda(x)=0$, then $\mathcal{B}u$ is reduced to a Neumann boundary condition.

If $D$ is a penetrable obstacle, then the scattering of  $u^{\rm inc}(x, x_s)$ by the scatterers $(\Gamma, D)$ can be formulated by
\be\label{a2}
\left\{\begin{array}{lll}
         \Delta u+\kappa^2n(x)u=0&\textrm{in}\;\; \R^2\se\{x_s\}, \\[2mm]
         \ds\lim_{r\rightarrow \infty}\sqrt{r}\left(\frac{\partial u^{s}}{\partial r}-{\rm i}\kappa u^{s}\right)=0\;\;\quad & {\rm for}\;\;r=|x|,
       \end{array}
\right.
\en
where the total field $u:=u^{\rm inc}+u^{s}$ in $\Omega_1$ and $u:=u^{s}$ in $\Omega_2$ with $u^{s}$ being the scattered field,  $n\in L^\infty$ is the refractive index such that ${\rm Re}(n)>0$, ${\rm Im}(n)\geq0$ and $n=1$ in $\R^2\se\ov{D}$.

Given the interface $\G$, the embedded obstacle $D$ or the refractive index $n$, and the incident field $u^{\rm inc}$, the forward problem is to determine the distribution of the scattered wave $u^{s}$ in 
$\R^2$. There exists lots of references in the literature on the well-posedness of  Problem (\ref{a1}) or (\ref{a2}), if $D=\emptyset$. We refer to \cite{SE10, MT06} for the variational method and
\cite{LYZ13, DTS03, ZS03} for the integral equation method with employing a generalized Fredholm theory (cf. \cite{SZ97, SZ00}). We also refer to \cite{AIL05,JC00, DEKPS08, RG08, GHKMS05, LLLL15,LYZZ21,LZ10, WP10} for the case of a planar surface $\G$ and an embedded
obstacle $D$. Different from the previous works, in the first part of the paper we will propose a novel technique to establish the existence of a unique solution to Problem (\ref{a1}) or (\ref{a2}) by transferring 
the interface scattering problem (\ref{a1}) without the embedded obstacle $D$ into an equivalent Lippmann-Schwinger type integral equation defined in a bounded domain for which the well-posedness follows from a 
direct application of the classical Fredholm theory. Then the existence of the solution of Problem (\ref{a1}) or (\ref{a2}) can be obtained 
by making use of the technique of background Green's function. The main advantage of the method is to avoid the discussion of related operators in unbounded domain and can further leads to the $L^p$ $(p>1)$
solution of Problem (\ref{a1}) or (\ref{a2}) when the incident filed is induced by a family of hypersingular point sources.

Compared with the forward problem, the inverse problem we are interested in is to determine the locally rough interface $\Gamma$, the wave number $\kappa_2$ and the embedded obstacle $(D,\mathcal{B})$ or 
the refractive index $n$ by taking the measurements of the scattered fields in the domain $\Om_1$. To the best of authors' knowledge, no results are available in the literature for the simultaneous recovery of the interface and embedded obstacles. If $\Gamma$ is just a planar interface with an embedded obstacle in the lower half-space, or $\Gamma$ is an impenetrable surface or a locally rough interface without embedded obstacle, lots of works have investigated for the inverse acoustic and electromagnetic problems; see \cite{AIL05,GL13,GL14,GJ11,CR10,CG11,JC00,DEKPS08,DLLY17,KG80, RG08,GHKMS05,LLLL15,LB13, LZ13,LWZ19,LYZZ21,LZ10,WP10,YL18} and reference therein. Especially, it was shown in \cite{SC95} that a uniqueness result was first established for a Dirichlet surface by using the incident plane waves, if the homogeneous medium above the surface is lossy. It was shown in \cite{LZ10} that an embedded electromagnetic obstacle was uniquely determined in a two-layered lossy medium separated by a planar surface with respect to incident point sources, while the similar uniqueness result was obtained in \cite{YL18} for inverse acoustic scattering by an embedded penetrable obstacle in the lower-half space for which the background medium was allowed to be non-lossy. Moreover, 
many numerical methods have been extensively studied for the inverse problems such as  the MUSIC-type method \cite{AIL05}, the algorithms based on transformed field expansions \cite{GL13,GL14},   Newton-type algorithms \cite{GJ11,ZZ13}, the Kirsch-Kress schemes \cite{CR10,LZ13},  qualitative methods \cite{DLLY17,AL08}, the asymptotic factorization method \cite{RG08}, the time-domain singular source method in \cite{C03}, and direct sampling methods \cite{LLLL15}.
Furthermore, a related uniqueness result was recently obtained in \cite{LWZ19} for recovering an interface and a perfectly conducting obstacle embedded in the upper-half space by assuming the two-layered lossy medium.

In the second part of this paper, we are motivated by \cite{YZZ13} for inverse obstacle scattering in a bounded inhomogeneous medium to investigate the unique determination of $(\G,D,\mathcal{B},\kappa_2)$ and 
$(\G,n,\kappa_2)$ by measurements of the scattered fields on a line segment of $\Om_1$. To this end, based on the technique proposed for the model problem, we first establish uniform  a priori estimates of solutions of 
Problem (\ref{a1}) or (\ref{a2}) when the incident waves are induced by a family of hyper-singular point sources $\partial_{x_\ell} \Phi_{\kappa_1}(\cdot,x_s)$ for $\ell =1,2$ with $x_s$ approaching the local perturbation of the interface. Then the two group of solutions of Problem (\ref{a1}) or (\ref{a2}) associated with two different locally interfaces are coupled in a sufficiently small domain as an interior transmission problem
(ITP). The uniform $L^2$-regularity of the solutions will be auxiliarily obtained as $x_s$ approaches the interface by the well-posedness of the ITP, which will leads to a contradiction. The unique recovery of $\G$ is thus 
obtained. One advantage of this technique is that the background medium of the embedded obstacle can be simultaneously recovered in view of the scattered field data in a similar way. Finally, the inverse problem is reduced to the case of recovering the embedded obstacle into a known layered medium and the uniqueness directly follows from the standard discussions (cf. \cite{ALB08, CK13}).

The remaining of this paper is built up as follows. In Section 2, we briefly introduce some necessary function spaces and  the background Green function associated with a planar surface. In Section 3, 
we present a novel technique for the well-posedness of the scattering problem in a two-layered medium by reducing the model problem into an equivalently Lippmann-Schwinger type integral equation in a bounded domain. 
In Section 4, we prove a global uniqueness theorem for the inverse problem of simultaneously recovering locally rough interfaces
and the embedded obstacles with its surrounding homogeneous medium.

\section{Preliminaries}
\setcounter{equation}{0}

\subsection{Some useful function spaces} 

Let $\Omega$ be a bounded domain of $\R^2$ with a Lipschitz boundary $\partial \Omega$. Let $W^{m,p}(\Omega)$ denote the usual Sobolev space with index $m\in{\mathbb N}$ and $p\in [1,\infty)$, equipped with 
the norm 
\ben
\|u\|_{m,p}:=\left(\sum\limits_{|\alpha|\leq m} \|\pa^{\alpha}u\|^p_{L^p(\Omega)}\right)^{1/p}.
\enn
For $p=2$, we also write $W^{m,2}(\Omega)$ by the notation $H^{m}(\Omega)$ which is a Hilbert space under the inner product 
\ben
(u,v)_m: = \sum_{|\alpha|\leq m}(\pa^\alpha u,\pa^\alpha v)_{L^2(\Omega)}.
\enn 
For $m=0$, $W^{0,p}(\Omega)$ is reduced to the usual $L^p(\Omega)$ space consisting of all $L^p$-integrable functions on $\Omega$. Moreover, we also introduce the following function space
\ben
H^1_{\Delta}(\Omega):=\{u\in\mathcal{D'}(\Om)|u\in H^1(\Om),\;\Delta u\in L^2(\Om)\}
\enn
which is a Hilbert space with respect to the inner product 
\ben
(u,v)_{H^1_{\Delta}(\Om)}=(u,v)_{L^2(\Om)}+(\nabla u,\nabla v)_{L^2(\Om)}+(\Delta u,\Delta v)_{L^2(\Om)}
\quad{\rm for\;}u,v\in H^1_{\Delta}(\Om),
\enn
where $\mathcal{D'}(\Om)$ denotes the set consisting of all distribution functions defined on $C^\infty_0(\Om)$.

\subsection{The background Green's function}
In this subsection, we introduce the two-dimensional background Green's function associated with the Helmholtz equation in a two-layered medium separated by the planar surface $\Gamma_0$. Suppose the incident wave $u^{\rm inc}(x, x_s)$ is induced by a point source $\Phi_{\kappa}(x, x_s)$ for $\kappa=\kappa_1$ or $\kappa_2$, which means $u^{\rm inc}(x,x_s)=\Phi_{\kappa_1}(x,x_s)$ for $x_s\in \R^2_+:=\{(x_1,x_2)\in\R^2: x_2>0\}$ and $u^{\rm inc}(x,x_s)=\Phi_{\kappa_2}(x,x_s)$ for $x_s\in \R^2_-:=\{(x_1,x_2)\in\R^2: x_2<0\}$. Then the scattering of $u^{\rm inc}(x, x_s)$ by the planar surface  
$\G_0$ can be modelled by 
\ben\label{b1}
\left\{\begin{array}{lll}
         \Delta_x {\mathbb G}(x, x_s; \Gamma_0)+\kappa_0^2(x){\mathbb G}(x, x_s; \Gamma_0)=-\delta(x-x_s) & \textrm{in}\;\R^2 \\ [2mm]
         \ds\lim_{r\rightarrow \infty}\sqrt{r}\left(\frac{\partial {\mathbb G}^{s}(x, x_s; \Gamma_0)}{\partial r}-{\rm i}\kappa_0{\mathbb G}^{s}(x, x_s; \Gamma_0)\right)=0& {\rm for}\; r=|x|
       \end{array}
\right.
\enn
where $\kappa_0>0$ is the wave number satisfying $\kappa_0(x):=\kappa_1>0$ for $x\in\R^2_+$ and $\kappa_0(x):=\kappa_2>0$ for $x\in\R^2_-$, and ${\mathbb G}^{s}(x,x_s;\Gamma_0)$ denotes the scattered field when the observation point $x$ and the source point $x_s$ belong to the same half-space, denotes the transmitted field when the observation point $x$ and the source point $x_s$ belong to the different half-space. For the relation between ${\mathbb G}$ and ${\mathbb G}^s$, we have 
\ben
{\mathbb G}(x,x_s;\Gamma_0)=\left\{\begin{array}{lll}
{\mathbb G}^{s}(x,x_s;\Gamma_0)+u^{\rm inc}(x,x_s)\quad{\rm for}\;\; x\in \R^2_+,\;\;x_s\in\R^2_+\;\;{\rm or}\;\; x\in \R^2_-,\;\;x_s\in\R^2_-,\\ [2mm]
{\mathbb G}^{s}(x,x_s;\Gamma_0)\qquad\qquad\qquad\;\;{\rm for}\;\; x\in \R^2_-,\;\;x_s\in\R^2_+\;\;{\rm or}\;\; x\in \R^2_+,\;\;x_s\in\R^2_-.\end{array}
\right.
\enn

Now we are in position to deduce the formulation of ${\mathbb G}^{s}(x,x_s;\Gamma_0)$.  It follows from the Fourier transform \cite{L10} that   
\ben
{\mathbb G}^s(x,x_s;\Gamma_0)=\left\{\begin{aligned}
\frac{\rm i}{4\pi}\int_{-\infty}^{+\infty}\frac{1}{\beta_1}\frac{\beta_1-\beta_2}{\beta_1+\beta_2}e^{{\rm i}\beta_1(x_2+x_{s2})}e^{{\rm i}\xi (x_1-x_{s1})}d\xi\quad\;\; {\rm for}\;\;x\in \R^2_+,\;\;x_s\in\R^2_+,\\[2mm]
\frac{\rm i}{2\pi}\int_{-\infty}^{+\infty}\frac{1}{\beta_1+\beta_2}e^{{\rm i}(\beta_1x_{s2}-\beta_2x_2)}e^{{\rm i}\xi (x_1-x_{s1})}d\xi\quad \;\;\;{\rm for}\;\;x\in \R^2_-,\;\;x_s\in\R^2_+,\\[2mm]
\frac{\rm i}{2\pi}\int_{-\infty}^{+\infty}\frac{1}{\beta_1+\beta_2}e^{{\rm i}(\beta_1x_2-\beta_2x_{s2})}e^{{\rm i}\xi (x_1-x_{s1})}d\xi\quad \;\;\;{\rm for}\;\;x\in \R^2_+,\;\;x_s\in\R^2_-,\\[2mm]
\frac{\rm i}{4\pi}\int_{-\infty}^{+\infty}\frac{1}{\beta_2}\frac{\beta_2-\beta_1}{\beta_1+\beta_2}e^{-{\rm i}\beta_2(x_2+x_{s2})}e^{{\rm i}\xi (x_1-x_{s1})}d\xi\quad {\rm for}\;\;x\in \R^2_-,\;\;x_s\in\R^2_-,
\end{aligned}
\right.
\enn
where $\beta_1$, $\beta_2$ are defined by 
\ben
\beta_j=\left\{\begin{array}{l}
                 \sqrt{\kappa_j^2-\xi^2}\quad\; {\rm for}\quad |\kappa_j|>|\xi|,\\
                 {\rm i}\sqrt{\xi^2-\kappa_j^2}\quad {\rm for}\quad |\kappa_j|<|\xi|,
               \end{array}
\right.
\enn
for $j=1,2$. 

To establish the uniqueness results for inverse problems, we need to consider the scattering of a hyper-singular point source $\partial_{x_\ell}\Phi_{\kappa}(x, x_s)$ for $\ell=1,2$ and $\kappa=\kappa_1,\kappa_2$,  more precisely, $u^{\rm inc}(x,x_s)=\partial_{x_\ell}\Phi_{\kappa_1}(x,x_s)$ for $x_s\in \R^2_+$ and $u^{\rm inc}(x,x_s)=\partial_{x_\ell}\Phi_{\kappa_2}(x,x_s)$ for $x_s\in \R^2_-$. For this case, the scattering of the incident wave $u^{\rm inc}(x,x_s)$ by $\Gamma_0$ can be formulated by 
\be\label{b2}
\left\{\begin{array}{lll}
         \Delta_x {\mathbb U}(x, x_s; \Gamma_0)+\kappa_0^2(x){\mathbb U}(x, x_s; \Gamma_0)=-\partial_{x_\ell}\delta(x-x_s) & \textrm{in}\;\R^2 \\ [2mm]
        \ds\lim_{r\rightarrow \infty}\sqrt{r}\left(\frac{\partial {\mathbb U}^{s}(x, x_s; \Gamma_0)}{\partial r}-{\rm i}\kappa_0{\mathbb U}^{s}(x, x_s; \Gamma_0)\right)=0& {\rm for}\; r=|x|.
       \end{array}
\right.
\en
Similarly, the total field ${\mathbb U}(x, x_s; \Gamma_0)$ can be decomposed as 
\ben
{\mathbb U}(x,x_s;\Gamma_0)=\left\{\begin{array}{lll}
{\mathbb U}^{s}(x,x_s;\Gamma_0)+u^{\rm inc}(x,x_s)\quad{\rm for}\;\; x\in \R^2_+,\;\;x_s\in\R^2_+\;\;{\rm or}\;\; x\in \R^2_-,\;\;x_s\in\R^2_-,\\ [2mm]
{\mathbb U}^{s}(x,x_s;\Gamma_0)\qquad\qquad\qquad\;\;{\rm for}\;\; x\in \R^2_-,\;\;x_s\in\R^2_+\;\;{\rm or}\;\; x\in \R^2_+,\;\;x_s\in\R^2_-.\end{array}
\right.
\enn
By means of the Fourier transform, we have: for $\ell=1$
\ben
{\mathbb U}^s(x,x_s;\Gamma_0)=\left\{\begin{aligned}
-\frac{1}{4\pi}\int_{-\infty}^{+\infty}\frac{1}{\beta_1}\frac{\beta_1-\beta_2}{\beta_1+\beta_2}\xi e^{{\rm i}\beta_1(x_2+x_{s2})}e^{{\rm i}\xi (x_1-x_{s1})}d\xi\quad\;\; {\rm for}\;\;x\in \R^2_+,\;\;x_s\in\R^2_+,\\[2mm]
-\frac{1}{2\pi}\int_{-\infty}^{+\infty}\frac{1}{\beta_1+\beta_2}\xi e^{{\rm i}(\beta_1x_{s2}-\beta_2x_2)}e^{{\rm i}\xi (x_1-x_{s1})}d\xi\quad \;\;\;{\rm for}\;\;x\in \R^2_-,\;\;x_s\in\R^2_+,\\[2mm]
-\frac{1}{2\pi}\int_{-\infty}^{+\infty}\frac{1}{\beta_1+\beta_2}\xi e^{{\rm i}(\beta_1x_2-\beta_2x_{s2})}e^{{\rm i}\xi (x_1-x_{s1})}d\xi\quad \;\;\;{\rm for}\;\;x\in \R^2_+,\;\;x_s\in\R^2_-,\\[2mm]
-\frac{1}{4\pi}\int_{-\infty}^{+\infty}\frac{1}{\beta_2}\frac{\beta_2-\beta_1}{\beta_2+\beta_1}\xi e^{-{\rm i}\beta_2(x_2+x_{s2})}e^{{\rm i}\xi (x_1-x_{s1})}d\xi\quad {\rm for}\;\;x\in \R^2_-,\;\;x_s\in\R^2_-,
\end{aligned}
\right.
\enn
and for $\ell=2$
\ben
{\mathbb U}^s(x,x_s;\Gamma_0)=\left\{\begin{aligned}
\frac{1}{4\pi}\int_{-\infty}^{+\infty}\frac{\beta_1-\beta_2}{\beta_1+\beta_2} e^{{\rm i}\beta_1(x_2+x_{s2})}e^{{\rm i}\xi (x_1-x_{s1})}d\xi\qquad\;\;\; {\rm for}\;\;x\in \R^2_+,\;\;x_s\in\R^2_+,\\[2mm]
\frac{1}{2\pi}\int_{-\infty}^{+\infty}\frac{\beta_1}{\beta_1+\beta_2}e^{{\rm i}(\beta_1x_{s2}-\beta_2x_2)}e^{{\rm i}\xi (x_1-x_{s1})}d\xi\quad \;\;\;\;{\rm for}\;\;x\in \R^2_-,\;\;x_s\in\R^2_+,\\[2mm]
-\frac{1}{2\pi}\int_{-\infty}^{+\infty}\frac{\beta_2}{\beta_1+\beta_2}e^{{\rm i}(\beta_1x_2-\beta_2x_{s2})}e^{{\rm i}\xi (x_1-x_{s1})}d\xi\quad \;{\rm for}\;\;x\in \R^2_+,\;\;x_s\in\R^2_-,\\[2mm]
-\frac{1}{4\pi}\int_{-\infty}^{+\infty}\frac{\beta_2-\beta_1}{\beta_2+\beta_1} e^{-{\rm i}\beta_2(x_2+x_{s2})}e^{{\rm i}\xi (x_1-x_{s1})}d\xi\quad\;\; {\rm for}\;\;x\in \R^2_-,\;\;x_s\in\R^2_-.
\end{aligned}
\right.
\enn
Using the dominated convergence theorem to ${\mathbb G}^s(x,x_s;\Gamma_0)$, ${\mathbb U}^s(x,x_s;\Gamma_0)$ and their derivatives shows that ${\mathbb G}^s(x,x_s;\Gamma_0)\in C^{\infty}(\R^2\se\{x_s\})$ and ${\mathbb U}^s(x,x_s;\Gamma_0)\in C^{\infty}(\R^2\se\{x_s\})$.

\section{The well-posedness in $L^p$ for $1<p<2$}
\setcounter{equation}{0}

This section is devoted to the well-posedness of the direct scattering problem (\ref{a1}) and (\ref{a2}). The uniqueness can be first obtained by a direct application of Proposition 2.1 in \cite{CH98}. 
To show the existence, a special case of $D=\emptyset$ in \eqref{a1} or $n(x)\equiv 1$ in \eqref{a2} will be first considered, which is equivalent to a Lippmann-Schwinger integral equation defined in a bounded domain by introducing a special interface. Based on this result, we then investigate the general case where $D\neq \emptyset$ or $n(x)\not\equiv 1$, which exists a solution by utilizing the integral equation method. In this way, we are able to construct a $L^p(1<p<2)$ estimate of the solution to the problem, which will play a central role in the proof of the global uniqueness result for the inverse problem in the next section.

\begin{theorem}\label{thm3.1}
{\rm The direct scattering problem (\ref{a1}) or (\ref{a2}) has at most one solution.}
\end{theorem}
\begin{proof}
Let $u^{\rm inc}=0$, from the proposition 2.1 in \cite{CH98}, it is sufficient to prove 
\be\label{c1}
\lim_{r\rightarrow \infty}\int_{\pa B_r}\left(\left|\frac{\pa u}{\pa \nu}\right|^2+|u|^2\right)ds=0,
\en
where $B_r:=\{x\in\R^2: |x|<r\}$.
The equation (\ref{c1}) is a direct consequence of combination of the Sommerfeld radiation condition, the Green's first theorem with the boundary condition on $\pa D$ or the assumption ${\rm Re}(n)>0$.
\end{proof}

\subsection{The special case $D=\emptyset$ or $n(x)\equiv 1$}

In this subsection, we study the scattering of point sources only by a locally rough interface $\Gamma$, which means that no obstacles are embedded into the two-layered medium. Different from all existing works, e.g., \cite{SE10, MT06, LYZ13, DTS03, ZS03}, we will propose a novel technique to prove the existence of a unique solution by transforming the reduced model  into  an equivalent Lippmann-Schwinger type integral equation defined in a bounded domain. Then the well-posedness in $L^p$ space  will directly follows from  the classical Fredholm theory. 

We consider two classes of incident waves
\begin{eqnarray*}
u^{\rm inc}(x,x_s)=\left\{\begin{aligned}&\Phi_{\kappa_1}(x,x_s)\qquad\;&{\rm for}\;\; x_s\in\Omega_1\\
&\partial_{x_{\ell}}\Phi_{\kappa_1}(x,x_s)\;\;\;&{\rm for}\;\; x_s\in\Omega_1
\end{aligned}
\right.
\end{eqnarray*}
where $\ell=1,2$. Compared with the incident wave $\Phi_{\kappa_1}(x,x_s)$, it is easily seen that $\partial_{x_{\ell}}\Phi_{\kappa_1}(x,x_s)$ is a hyper-singular point source at $x_s$.
To indicate the dependence of the total field on $u^{\rm inc}$, we use ${\mathbb G}(x,x_s;\Gamma)$ and ${\mathbb U}(x,x_s;\Gamma)$ to denote the total field corresponding to the scattering of $\Phi_{\kappa_1}(x,x_s)$ and $\partial_{x_{\ell}}\Phi_{\kappa_1}(x,x_s)$, respectively, by the locally rough surface $\Gamma$. In the following, for simplicity, we only provide a detailed proof for the existence of ${\mathbb U}(x,x_s;\Gamma)$. The existence of ${\mathbb G}(x,x_s;\Gamma)$ can be obtained in a similar manner. 

Note the total field ${\mathbb U}(x,x_s;\Gamma)$ satisfies
\be\label{c2}
\left\{\begin{array}{lll}
         \Delta_x {\mathbb U}(x,x_s;\Gamma)+\kappa_{\Gamma}^2(x){\mathbb U}(x,x_s;\Gamma)=-\partial_{x_\ell}\delta(x-x_s) \;\;\qquad \textrm{in}\;\;\R^2 \\[2mm]
         \ds\lim_{r\rightarrow \infty}\sqrt{r}\left(\frac{\partial {\mathbb U}^s(x,x_s;\Gamma)}{\partial r}-{\rm i}\kappa_{\Gamma}{\mathbb U}^s(x,x_s;\Gamma)\right)=0\qquad\qquad {\rm for}\; r=|x|
       \end{array}
\right.
\en
where ${\mathbb U}^s(x,x_s;\Gamma)$ denotes the scattered field defined as ${\mathbb U}^s(x,x_s;\Gamma):={\mathbb U}(x,x_s;\Gamma)-\partial_{x_{\ell}}\Phi_{\kappa_1}(x,x_s)$ in $\Omega_1$ and ${\mathbb U}^s(x,x_s;\Gamma):={\mathbb U}(x,x_s;\Gamma)$ in $\Omega_2$ for the case $x_s\in\Omega_1$. Here, $\kappa_{\Gamma}$ denotes the wave number satisfying $\kappa_{\Gamma}=\kappa_1$ in $\Omega_1$ and $\kappa_{\Gamma}=\kappa_2$ in $\Omega_2$.

The uniqueness of Problem (\ref{c2}) is a direct consequence of  Theorem \ref{thm3.1}. To show the existence of Problem (\ref{c2}), we first consider the scattering of  $u^{\rm inc}(x,x_s)$ by a special interface
\begin{equation*}
\Gamma_{R}:=\{(x_1,x_2)\in\R^2: x_2=0\;\;{\rm for}\;\;|x_1|\geq R\;\;{\rm and}\;\; x_2=-\sqrt{R^2-x_1^2}\;\;{\rm for}\;\;|x_1|<R\}.
\end{equation*} 
For $u^{\rm inc}(x,x_s)=\partial_{x_{\ell}}\Phi_{\kappa_1}(x,x_s)$ with $x_s\in\Omega_{1,R}$, one aims to find the total field 
${\mathbb U}(x,x_s;\Gamma_R):=\partial_{x_{\ell}}\Phi_{\kappa_1}(x,x_s)+{\mathbb U}^s(x,x_s;\Gamma_R)$ in $\Omega_{1,R}$ and ${\mathbb U}(x,x_s;\Gamma_R):={\mathbb U}^{s}(x,x_s;\Gamma_R)$ in $\Omega_{2,R}$ which solves 
\be\label{c4}
\left\{\begin{array}{lll}
         \Delta_x {\mathbb U}(x,x_s;\Gamma_R)+\kappa_{R}^2(x){\mathbb U}(x,x_s;\Gamma_R)=-\partial_{x_\ell}\delta(x-x_s) \qquad \textrm{in}\;\;\R^2 \\[2mm]
         \ds\lim_{r\rightarrow \infty}\sqrt{r}\left(\frac{\partial {\mathbb U}^s(x,x_s;\Gamma_R)}{\partial r}-{\rm i}\kappa_{R}{\mathbb U}^s(x,x_s;\Gamma_R)\right)=0\qquad\quad\;\; {\rm for}\quad r=|x|.
       \end{array}
\right.
\en
Here, $\Omega_{1,R}$ and $\Omega_{2,R}$ denote the upper and lower half-spaces separated by the locally rough surface $\Gamma_{R}$, and $\kappa_{R}$ is the wave number satisfying $\kappa_R=\kappa_1$ in $\Omega_{1,R}$ and $\kappa_R=\kappa_2$ in $\Omega_{2,R}$.

Recall that ${\mathbb U}(x,x_s;\Gamma_0)$ is the total field of Problem (\ref{b2}). Define the difference 
\be\label{c3}
V(x,x_s):={\mathbb U}(x,x_s;\Gamma_R)-{\mathbb U}(x,x_s;\Gamma_0),
\en
which satisfies the following problem 
\be\label{c5}
\left\{\begin{array}{lll}
         \Delta_x V(x,x_s)+\kappa_0^2V(x,x_s)=\varphi_1(x) \quad\qquad\;\;\; \textrm{in}\;\;\R^2, \\[2mm]
         \ds\lim_{r\rightarrow \infty}\sqrt{r}\left(\frac{\partial V(x,x_s)}{\partial r}-{\rm i}\kappa_0V(x,x_s)\right)=0\quad {\rm for}\quad r=|x|
       \end{array}
\right.
\en
with the right term $\varphi_1$ given by 
\begin{eqnarray*}
\varphi_1(x):=\left\{\begin{aligned}&\eta {\mathbb U}(x,x_s;\Gamma_R)\qquad{\rm in}\;\; B_1\\
&0\qquad\qquad\qquad\quad\;{\rm in}\;\;\R^2\se{\ov{B_1}}.
\end{aligned}
\right.
\end{eqnarray*}
Here, $\eta:=\kappa_2^2-\kappa_1^2$ and 
$B_1:=\{(x_1,x_2)\in\R^2: -\sqrt{R^2-x_1^2}<x_2<0\;\;{\rm for}\;\;|x_1|<R\}.$
\begin{theorem}\label{thm1}
Let $u^{\rm inc}(x, x_s)=\partial_{x_{\ell}}\Phi_{\kappa_1}(x,x_s)\in L^{p}_{\rm loc}(\R^2)$ for $1<p<2$ and $x_s\in \R^2_+$. If ${\mathbb U}^s(x,x_s;\Gamma_R)\in W^{2,p}_{\rm loc}(\R^2)$ is the scattered field associated with Problem (\ref{c4}), then ${\mathbb U}(x,x_s;\Gamma_R)|_{B_1}:=({\mathbb U}^s(x,x_s;\Gamma_R)+u^{\rm inc}(x,x_s))|_{B_1}$ is a solution to the following Lippmann-Schwinger equation 
\begin{equation}\label{c6} 
{\mathbb U}(x,x_s;\Gamma_R)+\eta\int_{B_1}{\mathbb G}(x,y,\Gamma_0){\mathbb U}(y,x_s;\Gamma_R)dy={\mathbb U}(x,x_s;\Gamma_0),\quad x\in B_1.
\end{equation}
Conversely, if ${\mathbb U}(x,x_s;\Gamma_R)|_{B_1}\in L^{p}(B_1)$ is a solution to the Lippmann-Schwinger equation (\ref{c6}), then ${\mathbb U}^s(x,x_s;\Gamma_R):={\mathbb U}(x,x_s;\Gamma_R)-u^{\rm inc}(x,x_s)$ can be extended to a solution to Problem (\ref{c4}) such that ${\mathbb U}^s(x,x_s;\Gamma_R)\in W_{\rm loc}^{2,p}(\R^2).$
\end{theorem}
\begin{proof}
Let $x\in B_1$ be an arbitrary point, we choose a sufficient small $\varepsilon>0$ such that the ball $B_{\varepsilon}(x)$ with $x$ as the centre and $\varepsilon$ as the radius contains in the domain $B_1$. If ${\mathbb U}^s(x,x_s;\Gamma_R)$ is the scattered field associated to Problem (\ref{c4}), applying the second Green's theorem to $V(y,x_s)$ and ${\mathbb G}(y,x;\Gamma_0)$ in the domain $B_1\se B_{\varepsilon}(x)$ yields
\be\no
&&\int_{B_1\setminus B_\varepsilon (x)}\left(\Delta V(y,x_s){\mathbb G}(y,x;\Gamma_0)-\Delta {\mathbb G}(y,x;\Gamma_0)V(y,x_s)\right)dy\\\no
&&=\left\{\int_{\G_0\setminus\G_R}-\int_{\G_R\setminus\G_0}-\int_{\partial B_\varepsilon (x)}\right\}
\left(\frac{\partial V(y,x_s)}{\pa \nu(y)}{\mathbb G}(y,x;\Gamma_0)-\frac{\partial {\mathbb G}(y,x;\Gamma_0)}{\pa \nu(y)}V(y,x_s)\right)ds(y)\\\label{c7}
&&:=I_1-I_2-I_3
\en
where $\nu(y)$ denotes the upward unit normal vector when $y\in\Gamma_0$ or $y\in\Gamma_R$, and denotes the outward unit normal vector when $y\in \partial B_{\varepsilon}(x)$. Employing (\ref{b2}) and (\ref{c5}) implies that the left hand side of (\ref{c7}) will trend to 
\begin{equation}\label{c8}
\eta\int_{B_1}{\mathbb G}(y,x;\Gamma_0){\mathbb U}(y,x_s;\Gamma_R)dy
\end{equation}
as $\varepsilon\to 0$. Substituting ${\mathbb G}(y,x;\Gamma_0)=\Phi_{\kappa_2}(y,x)+{\mathbb G}^s(y,x;\Gamma_0)$ into the term $I_3$ gives that $I_3=I_{31}+I_{32}$ with definitions
\ben
&&I_{31}:=\int_{\partial B_\varepsilon (x)}\left(\frac{\pa V(y,x_s)}{\pa \nu(y)}\Phi_{\kappa_2}(y,x)-\frac{\pa \Phi_{\kappa_2}(y,x)}{\pa \nu(y)}V(y,x_s)\right)ds(y),\\
&&I_{32}:=\int_{\partial B_\varepsilon (x)}\left(\frac{\pa V(y,x_s)}{\pa \nu(y)}{\mathbb G}^s(y,x;\Gamma_0)-\frac{\pa {\mathbb G}^s(y,x;\Gamma_0)}{\pa \nu(y)}V(y,x_s)\right)ds(y).
\enn
A direct manipulation, using the mean value theorem, shows that 
\be\label{c9}
\lim_{\varepsilon\to 0} I_{31}=V(x, x_s).
\en
A straightforward derivation using the equations that $V(y,x_s)$ and ${\mathbb G}^s(y,x;\Gamma_0)$ satisfy gives that 
\be\label{c10}
I_{32}=\eta\int_{B_{\varepsilon}(x)}{\mathbb G}^s(y,x;\Gamma_0){\mathbb U}(y,x_s;\Gamma_R)dy\to 0\quad{\rm as}\quad \varepsilon\to 0,
\en
where we use the smooth of ${\mathbb U}(y,x_s;\Gamma_R)$ and ${\mathbb G}^s(y,x;\Gamma_0)$ for $x\in B_1$, $y\in B_{\varepsilon}(x)$ and $x_s\in \R^2_+$. Hence, (\ref{c9}) and (\ref{c10}) yield
\be\label{c11}
\lim_{\varepsilon\to 0} I_{3}=V(x, x_s).
\en
For $I_1$ and $I_2$, let $B_{R'}$ be the circle of radius $R'>R$ and center at the origin, and let $B^{\pm}_{R'}$ denote the upper and lower semi-circle, respectively. Application of the second Green's theorem for $G$ and ${\mathbb G}_0$ in the domain $B_{R'}^+$ and $B^-_{R'}\se B_1$ leads to 
 \ben\label{c12}
 I_1-I_2&=&\int_{\partial B_{R'}}
\bigg[\left(\frac{\pa V(y,x_s)}{\pa \nu(y)}-{\rm i}\kappa_0V(y,x_s)\right){\mathbb G}(y,x;\Gamma_0)\\
&&\qquad-\left(\frac{\pa {\mathbb G}(y,x;\Gamma_0)}{\pa \nu(y)}-{\rm i}\kappa_0{\mathbb G}(y,x;\Gamma_0)\right)V(y,x_s)\bigg]ds(y)
\enn
where $\nu(y)$ stands for the outward unit normal vector when $y\in \pa B_{R'}$. To prove (\ref{c6}), we first show that 
\be\label{c13}
\int_{\partial B_{R'}}|V(y,x_s)|^2ds(y)=O(1),\quad \int_{\partial B_{R'}}|{\mathbb G}(y,x;\Gamma_0)|^2ds(y)=O(1),\quad R'\rightarrow\infty.
\en
It follows from the Sommerfeld radiation condition that 
\be\label{c14}
\int_{\partial B_{R'}^+}\left[\left|\frac{\pa V(y,x_s)}{\pa \nu(y)}\right|^2+\kappa_1^2|V(y,x_s)|^2+2\kappa_1{\rm Im}\left(\frac{\pa\overline{V(y,x_s)}}{\pa \nu(y)}V(y,x_s)\right) \right]ds(y)\to 0,
\en
\be\label{c15}
\int_{\partial B_{R'}^-}\left[\frac{\kappa_1}{\kappa_2}\left|\frac{\pa V(y,x_s)}{\pa \nu(y)}\right|^2+\kappa_1\kappa_2|V(y,x_s)|^2+2\kappa_1{\rm Im}\left(\frac{\pa\overline{V(y,x_s)}}{\pa \nu(y)}V(y,x_s)\right) \right]ds(y)\to 0,
\en
as $R'\rightarrow\infty$. 
Applying the Green's first theorem for $V$ and $\overline{V}$ in $B_{R'}^+$ and $B_{R'}^-\se\ov{B_1}$ implies
\be\no
&&\int_{\partial B_{R'}}\frac{\pa\overline{V(y,x_s)}}{\pa \nu(y)}V(y,x_s) ds(y)
=\left\{\int_{\G_0\setminus\G_R}-\int_{\G_R\setminus\G_0}\right\}\frac{\pa\overline{V(y,x_s)}}{\pa \nu(y)}V(y,x_s)ds(y)\\\label{c16}
&&+\int_{B_{R'}\se\ov{B_1}}\left(|\nabla V(y,x_s)|^2-\kappa_0^2|V(y,x_s)|^2\right)dy
\en
We now insert the imaginary part of (\ref{c16}) into (\ref{c14})-(\ref{c15}) and find that
\ben
&&\int_{\partial B_{R'}^+}\left(\left|\frac{\pa V(y,x_s)}{\pa \nu(y)}\right|^2+\kappa_1^2|V(y,x_s)|^2 \right)ds(y)\\
&&+\int_{\partial B_{R'}^-}\left(\frac{\kappa_1}{\kappa_2}\left|\frac{\pa V(y,x_s)}{\pa \nu(y)}\right|^2+\kappa_1\kappa_2|V(y,x_s)|^2\right)ds(y) \\
&&\rightarrow 2\kappa_1{\rm Im}\left\{\int_{\G_R\setminus\G_0}-\int_{\G_0\setminus\G_R}\right\}\frac{\pa\overline{V(y,x_s)}}{\pa \nu(y)} V(y,x_s)ds(y),\quad R'\to\infty,
\enn
which implies the first part in (\ref{c13}) holds. Similarly, we can prove that the second part in  (\ref{c13}) holds.
Thus, it follows from the Cauchy-Schwarz inequality using the Sommerfeld radiation condition and (\ref{c13}) that 
$I_1-I_2=0$. Hence, using (\ref{c7}), (\ref{c8}), (\ref{c11}) and the definition of $V(x, x_s):={\mathbb U}(x,x_s;\Gamma_R)-{\mathbb U}(x,x_s;\Gamma_0)$ shows that ${\mathbb U}(x,x_s;\Gamma_R)|_{B_1}$ is a solution 
\begin{equation*}\label{c17} 
{\mathbb U}(x,x_s;\Gamma_R)+\eta\int_{B_1}{\mathbb G}(y,x;\Gamma_0){\mathbb U}(y,x_s;\Gamma_R)dy={\mathbb U}(x,x_s;\Gamma_0)\quad x\in B_1.
\end{equation*}
It is easily verified from the formulation of ${\mathbb G}(\cdot,\cdot,\Gamma_0)$ that the reciprocity relation ${\mathbb G}(y,x;\Gamma_0)={\mathbb G}(x,y;\Gamma_0)$ holds for $x,y\in B_1, x\not=y$, thus, we conclude that ${\mathbb U}(x,x_s;\Gamma_R)|_{B_1}$ is a solution to (\ref{c6}).

Conversely, let ${\mathbb U}(x,x_s;\Gamma_R)\in L^p(B_1)$ be a solution of (\ref{c6}) and define ${\mathbb U}(x,x_s;\Gamma_R)$ by
\begin{equation}\label{c18} 
{\mathbb U}(x,x_s;\Gamma_R):={\mathbb U}(x,x_s;\Gamma_0)-\eta\int_{B_1}{\mathbb G}(x,y;\Gamma_0){\mathbb U}(y,x_s;\Gamma_R)dy\quad x\in \R^2
\end{equation}
which gives that the scattered field ${\mathbb U}^s(x,x_s;\Gamma_R):={\mathbb U}(x,x_s;\Gamma_R)-u^{\rm inc}(x,x_s)$ in $\Omega_{1,R}$ and ${\mathbb U}^s(x,x_s;\Gamma_R):={\mathbb U}(x,x_s;\Gamma_R)$ in $\Omega_{2,R}$ belongs to $W_{\rm loc}^{2,p}(\R^2)$ from the smoothness of ${\mathbb G}(x, x_s;\Gamma_0)$ for $x\not=x_s$ and \cite{GT83}. In $\R^2_+$, it is easy to see $\Delta_x {\mathbb U}(x,x_s;\Gamma_R)+\kappa_1^2{\mathbb U}(x,x_s;\Gamma_R)=-\partial_{x_\ell}\delta(x-x_s)$ in $\R^2_+$; in the domain $B_1$, we have $\Delta_x {\mathbb U}(x,x_s;\Gamma_R) +\kappa_2^2{\mathbb U}(x,x_s;\Gamma_R)=\eta {\mathbb U}(x,x_s;\Gamma_R)=(\kappa_2^2-\kappa_1^2){\mathbb U}(x,x_s;\Gamma_R)$ which gives $\Delta_x {\mathbb U}(x,x_s;\Gamma_R)+\kappa_1^2{\mathbb U}(x,x_s;\Gamma_R)=0$ in $B_1$; and in the domain $\Omega_{2,R}$, we conclude $\Delta_x {\mathbb U}(x,x_s;\Gamma_R)+\kappa_2^2{\mathbb U}(x,x_s;\Gamma_R)=0$; thus, we have ${\mathbb U}(x,x_s;\Gamma_R)$ satisfies the equation $\Delta_x {\mathbb U}(x,x_s;\Gamma_R)+\kappa_R^2{\mathbb U}(x,x_s;\Gamma_R)=-\partial_{x_\ell}\delta(x-x_s)$ in $\R^2$. Since ${\mathbb G}(x,y;\Gamma_0)$ satisfies the Sommerfeld radiation condition, we conclude that ${\mathbb U}^s(x,x_s;\Gamma_R)$ also satisfies the Sommerfeld radiation condition. So ${\mathbb U}(x,x_s;\Gamma_R)$ given by (\ref{c18}) is a solution to Problem (\ref{c2}). The proof is finished.
\end{proof}

With the equivalence theorem \ref{thm1}, we can transform the well-posedness of Problem (\ref{c4}) into
solving the Lippmann-Schwinger type equation (\ref{c6}) in $L^p(B_1)$. To this end, we define the integral operator
$T_0: L^p(B_1)\rightarrow  L^p(B_1)$ by
\ben
T_0\varphi(x):=\int_{B_1}{\mathbb G}(x,y;\Gamma_0)\varphi(y)dy,
\enn
thus, the equation (\ref{c6}) can be rewritten in the following operator form
\be\label{c19}
(I+\eta T_0){\mathbb U}(\cdot,x_s;\Gamma_R)= {\mathbb U}(\cdot,x_s;\Gamma_0) \quad {\rm in}\quad L^p(B_1).
\en
where $I:L^p(B_1)\rightarrow L^p(B_1)$ is the identity operator. Now we are able to
obtain the following existence result for the  Lippmann-Schwinger type equation (\ref{c19}).
\begin{theorem}\label{thm2}
For $1< p < 2$, there exists a unique solution ${\mathbb U}(\cdot,x_s;\Gamma_R)\in L^p(B_1)$ to (\ref{c19}) such that 
\be\label{c20}
\|{\mathbb U}(\cdot,x_s;\Gamma_R)\|_{L^p(B_1)}\leq C_1\|{\mathbb U}(\cdot,x_s;\Gamma_0)\|_{L^p(B_1)}.
\en
\end{theorem}
\begin{proof}
By the fact that the operator $T_0: L^p(B_1)\rightarrow W^{2,p}(B_1)$ is bounded and the Sobolev compact embedding theorem, we conclude that $T_0: L^p(B_1)\rightarrow L^p(B_1)$ is compact. Hence, $I+\eta T_0:  L^p(B_1)\rightarrow  L^p(B_1)$ is a Fredholm operator. By the Riesz-Fredholm theory, the existence of a solution to (\ref{c19}) can be established from the uniqueness of (\ref{c19}). Let $(I+\eta T_0)\varphi=0$,  then 
\be\label{c21}
\varphi(x)=-\eta\int_{B_1}{\mathbb G}(x,y;\Gamma_0)\varphi(y)dy\quad {\rm for} \; x\in B_1, 
\en
which implies that $\Delta\varphi+\kappa_1^2\varphi=0$ in $B_1$. Furthermore, we can extend $\varphi$ to $\R^2\se B_1$ by the right hand side of (\ref{c21}). Thus, $\varphi$ satisfies the Helmholtz equation $\Delta\varphi+\kappa_{R}^2\varphi=0$ in $\R^2$ and Sommerfeld radiation condition. Then we conclude $\varphi=0$ from the uniqueness of Problem (\ref{c4}) which is a direct consequence of theorem {\ref{thm3.1}}. Thus, the operator $I+\eta T_0:  L^p(B_1)\rightarrow  L^p(B_1)$ is injective, so it is bijective from the Riesz-Fredholm theory and has a bounded inverse which implies that (\ref{c20}) holds. The proof is completed.
\end{proof}

Based on Theorem \ref{thm1}, Theorem \ref{thm2}, and (\ref{c3}), we have obtained the existence of Problem (\ref{c4}).
For Problem (\ref{c2}),  we now choose a sufficient large $R$ such that the local perturbation of $\Gamma$ lies totally above the local perturbation of $\Gamma_R$. For a fixed $x_s\in \Omega_1$, we consider the difference
\ben
 W(x,x_s):={\mathbb U}(x,x_s;\Gamma)-{\mathbb U}(x,x_s;\Gamma_R)
\enn
which satisfies 
\be\label{c24}
\left\{\begin{array}{lll}
         \Delta_x W(x,x_s)+\kappa_R^2W(x,x_s)=\varphi_2(x) \;\quad\qquad\;\; \textrm{in}\;\;\R^2 \\[2mm]
         \ds\lim_{r\rightarrow \infty}\sqrt{r}\left(\frac{\partial W(x,x_s)}{\partial r}-{\rm i}\kappa_RW(x,x_s)\right)=0\quad {\rm for}\; r=|x|
       \end{array}
\right.
\en
with 
\begin{eqnarray*}
\varphi_2(x):=\left\{\begin{aligned}&-\eta {\mathbb U}(x,x_s;\Gamma)\qquad{\rm in}\;\; B_2\\
&0\qquad\qquad\qquad\quad\;\;\;{\rm in}\;\;\R^2\se{\ov{B_2}},
\end{aligned}
\right.
\end{eqnarray*}{s}
and 
$B_2:=\{(x_1,x_2)\in\R^2: -\sqrt{R^2-x_1^2}<x_2<f(x_1)\;\;{\rm for}\;\;|x_1|<R\}.$

By similar arguments as in Theorem \ref{thm1} for Problem (\ref{c4}), it is not difficult to find that Problem (\ref{c24}) can be equivalently formulated as the following Lippmann-Schwinger equation defined in $B_2$. 
\begin{theorem}\label{thm3}
Let $u^{\rm inc}(x, x_s)=\partial_{x_{\ell}}\Phi_{\kappa_1}(x,x_s)\in L^{p}_{\rm loc}(\R^2)$ for $1<p<2$ and fixed $x_s\in \Omega_1$. If ${\mathbb U}^s(x,x_s;\Gamma)\in W^{2,p}_{\rm loc}(\R^2)$ is the scattered field associated with Problem (\ref{c2}), then ${\mathbb U}(x,x_s;\Gamma)|_{B_2}:=({\mathbb U}^s(x,x_s;\Gamma)+u^{\rm inc}(x,x_s))|_{B_2}$ is a solution to the following Lippmann-Schwinger equation 
\begin{equation}\label{c25} 
{\mathbb U}(x,x_s;\Gamma)-\eta\int_{B_2}{\mathbb G}(x,y;\Gamma_R){\mathbb U}(y,x_s;\Gamma)dy={\mathbb U}(x,x_s;\Gamma_R)\quad x\in B_2.
\end{equation}
Conversely, if ${\mathbb U}(x,x_s;\Gamma)|_{B_2}\in L^{p}(B_2)$ is a solution to the Lippmann-Schwinger equation (\ref{c25}), then ${\mathbb U}^s(x,x_s;\Gamma):={\mathbb U}(x,x_s;\Gamma)-u^{\rm inc}(x,x_s)$ can be extended to a solution to Problem (\ref{c2}) such that ${\mathbb U}^s(x,x_s;\Gamma)\in W_{\rm loc}^{2,p}(\R^2).$
\end{theorem}

Now we arrive at the position to present the main well-posedness result for the scattering problem without embedded obstacles.
\begin{theorem}\label{thm4}
For $1< p < 2$, there exists a unique solution ${\mathbb U}(\cdot,x_s;\Gamma)\in L^p(B_2)$ to (\ref{c25}) such that 
\be\label{c26}
\|{\mathbb U}(\cdot,x_s;\Gamma)\|_{L^p(B_2)}\leq C_2\|{\mathbb U}(\cdot,x_s;\Gamma_R)\|_{L^p(B_2)}.
\en
\end{theorem}

\begin{proof}

We here omit the proof, since it is analogous to Theorem \ref{thm2}. 
\end{proof}

\subsection{The general case $D\neq \emptyset$ or $n(x)\not\equiv 1$}

Based on the well-posedness of Problem (\ref{c2}), we are able to deal with the general case when $D\neq\emptyset$ or $n(x)\not\equiv 1$. Similarly, we only consider the case where the incident wave is taken as hyper-singular point sources. To this end, we define the difference  
\ben
 w(x,x_s):=u(x,x_s)-{\mathbb U}(x,x_s;\Gamma)\qquad {\rm in\;}\R^2\se D.
\enn
For an impenetrable obstacle $D$,  it is easy to see that $w$ satisfies
\be\label{c27}
\left\{\begin{array}{lll}
         \Delta_x w(x,x_s)+\kappa^2w(x,x_s)=0 \qquad\qquad\qquad \textrm{in}\;\; \R^2\se\ov{D} \\[2mm]
          \mathcal{B}w(x,x_s)=-\mathcal{B}{\mathbb U}(x,x_s;\Gamma)\;\qquad\qquad\qquad\;\;\textrm{on}\;\; \pa{D}\\[2mm]
         \ds\lim_{r\rightarrow \infty}\sqrt{r}\left(\frac{\partial w(x,x_s)}{\partial r}-{\rm i}\kappa w(x,x_s)\right)=0\quad\;\; {\rm for\;\;}r=|x|
       \end{array}
\right.
\en
by combining Problems (\ref{a1}) and (\ref{c2}).

If $D$ is a penetrable obstacle, it follows that $w$ satisfies 
 \be\label{c30}
\left\{\begin{array}{lll}
         \Delta_x w(x,x_s)+\kappa^2n(x)w(x,x_s)=\varphi_3(x) \qquad\qquad\;\; \textrm{in}\;\; \R^2 \\[2mm]
         \ds\lim_{r\rightarrow \infty}\sqrt{r}\left(\frac{\partial w(x,x_s)}{\partial r}-{\rm i}\kappa w(x,x_s)\right)=0\qquad\qquad {\rm for\;\;}r=|x|,
       \end{array}
\right.
\en
with $\varphi_3(x):=\kappa^2(1-n(x)){\mathbb U}(x,x_s;\Gamma)$. Thus, the well-posedness of Problem (\ref{a1}) or (\ref{a2}) can be now reduced to the case of Problem (\ref{c27}) or (\ref{c30}).
\begin{theorem}\label{thm5}
Problem (\ref{a1}) admits a unique solution.
\end{theorem}
\begin{proof}
The uniqueness of (\ref{c27}) follows immediately from Theorem \ref{thm3.1}. For the existence of the solution, we only consider the case of a sound-soft embedded obstacle $D$. 
Other cases can be dealt with in a similar manner with a slight modification. In this case, one tries to seek a solution in the form of a combined acoustic double- and single-layer potential
\be\label{c28}
w(x,x_s) = \int_{\pa D}\left(\frac{\pa {\mathbb G}(x,y;\Gamma)}{\pa \nu(y)}-{\rm i} {\mathbb G}(x,y;\Gamma)\right)\psi(y)ds(y) \quad \rm{for}\quad x\in\R^2\se\ov{D}
\en
with some unknown density $\psi\in L^p(\pa{D})$ for $p>1$. 
Then from the boundary condition imposed on $\partial D$, we see that the potential (\ref{c28}) solve Problem (\ref{c27}), provided the density $\psi$ is a solution of a second-kind integral equation 
\ben
(I+K-{\rm i}S)\psi = -2{\mathbb U}(x,x_s;\Gamma)\quad{\rm on}\;\partial D,
\enn
where $S$ and $K$ are the single- and double-layer operators given by
\ben
&&(S\psi)(x):=2\int_{\partial D}{\mathbb G}(x,y;\Gamma)\psi(y)ds(y)\qquad\quad x\in\partial D\\
&&(K\psi)(x):=2\int_{\partial D}\frac{\pa {\mathbb G}(x,y;\Gamma)}{\pa \nu(y)}\psi(y)ds(y)\qquad x\in\partial D,
\enn
respectively. Since $\pa D\in C^2$, we deduce by Lemma 1 in \cite{RP01}  that both $S$ and $K$ are bounded and compact operators in $L^p(\pa D)$. It is thus
concluded that $I+K-{\rm i}S: L^p(\pa{D})\rightarrow L^p(\pa{D})$ is a Fredholm operator with index $0$. It then follows from the arguments for Theorem 3.11 in \cite{CK13}, the existence of a density $\psi$ can be established with the aid of the Riesz-Fredholm theory. Furthermore, we also have the following estimate 
\be\label{c29}
\|\psi\|_{L^p(\pa D)}\leq C\|{\mathbb U}(x,x_s;\Gamma)\|_{L^p(\pa D)}.
\en
The proof is thus complete. 
\end{proof}

For Problem (\ref{a2}), we can also establish the following unique solvability. 
\begin{theorem}\label{thm6}
Problem (\ref{a2}) admits a unique solution.
\end{theorem}

\begin{proof}
With the aid of Theorem \ref{thm3.1}, we only need to show the existence of a solution of Problem (\ref{a2}). Note that  $w(x, x_s):=u(x, x_s)-{\mathbb U}(x,x_s;\Gamma)$ satisfies Problem (\ref{c30}). 
By an application of Green's theorems, it is derived that (\ref{c30}) is equivalent to the following Lippmann-Schwinger equation 
\ben\label{c31}
w(x, x_s)=-\kappa^2\int_{\R^2}{\mathbb G}(x,y;\Gamma)m(y)(w(y,x_s)+{\mathbb U}(x,y;\Gamma))dy
\enn
where $m:=1-n$.

Recalling $w(x, x_s)=u(x, x_s)-{\mathbb U}(x,y;\Gamma)$ in $\R^2$ gives 
\be\label{c32}
u(x, x_s)+\kappa^2\int_{\R^2} {\mathbb G}(x,y;\Gamma)m(y)u(y,x_s)dy={\mathbb U}(x,x_s;\Gamma).
\en
By the similar arguments with Theorem \ref{thm2}, we obtain that (\ref{c32}) is uniquely solvable and 
\begin{eqnarray*}
\|u(\cdot,x_s)\|_{L^p(D)}\leq C\|{\mathbb U}(\cdot,x_s;\Gamma)\|_{L^p(D)}.
\end{eqnarray*}
\end{proof}

\section{Uniqueness of the inverse problem}
\setcounter{equation}{0}
Based on the well-posedness of Problem (\ref{a1}) and (\ref{a2}) in Section 3, we investigate in this section the inverse problem of recovering both the locally rough surface $\G$ and the embedded obstacle $D$ (or the refractive index $n$) with its surrounding homogeneous medium $\kappa_2$ from the knowledge of scattered near-field data, associated with incident point sources $u^{\rm inc}(x,x_s) = \Phi_{\kappa_1}(x,x_s)$, taken only on a line segment denoted by 
\ben
\Gamma_{b,a}:=\left\{(x_1,x_2)\in\R^2: \;|x_1|\leq a,\;x_2=b,\; b>\|f\|_{L^\infty(\R)}\right\}.
\enn
Motivated by \cite{YZZ13}, we establish two global uniqueness theorems for determining all unknowns in the inverse problems. Our method is based on two ingredients, one is constructing an interior transmission problem in a small domain, which is uniquely solvable when the domain is chosen to be small enough; the other is a priori estimate of solutions to the transmission problem with data in $L^p(1<p<2)$, where the priori estimate depends on the well-posedness of the direct scattering problem in the previous section.  

For convenience, let $u_j^s(x,x_s)$, $j=1,2$, denote the scattered field in $\Om_1$ and $u_j(x,x_s)$ denote the transmitted field in $\Om_2$ associated with Problem (\ref{a1}) (or (\ref{a2})) with the complex scatterer 
$(\Gamma_j,\kappa_{2,j}, D_j,\mathcal{B}_j)$(or $(\Gamma_j, \kappa_{2,j}, n_j)$).

\subsection{The case of an impenetrable obstacle $D$}
In this subsection, $D$ is considered to be an impenetrable obstacle with the physical property $\mathcal{B}$ which is embedded into the lower-half space $\Om_2$.

\begin{theorem}\label{thm7}
If $u_1^s(x,x_s)=u_2^s(x,x_s)$ for all $x,x_s\in\Gamma_{b,a}$, then $(\Gamma_1,\kappa_{2,1}, D_1,\mathcal{B}_1)=(\Gamma_2,\kappa_{2,2}, D_2,\mathcal{B}_2)$.
\end{theorem}

\begin{proof}
Our first goal is to establish a generalized reciprocity relation for the solutions of Problem (\ref{a1}) corresponding to two classes of different incident point sources. More precisely, 
let $u(\cdot, x_s^1)$ and $u(\cdot, x_s^2)$ be the solutions to Problem (\ref{a1}) with the incident fields $u^{\rm inc}(\cdot,x_s^1)$ and $u^{\rm inc}(\cdot, x_s^2)$, respectively, for $x_s^1,x_s^2\in \R^2\se(\ov{D}\cup\G)$.
Define two functions $E_\varepsilon(z_1)$ and $F_\varepsilon(z_2)$ by
\be\label{e1}
&&E_\varepsilon(x_s^1):=\int_{\pa B_\varepsilon(x_s^1)}\left[\frac{\pa u^{\rm inc}(y,x_s^1)}{\partial \nu(y)}u(y, x_s^2)-\frac{\pa u(y,x_s^2)}{\partial \nu(y)}u^{\rm inc}(y,x_s^1)\right]{\rm d}s(y)\\[2mm]\label{e2}
&&F_\varepsilon(x_s^2):=\int_{\pa B_\varepsilon(x_s^2)}\left[\frac{\pa u^{\rm inc}(y,x_s^2)}{\partial \nu(y)}u(y, x_s^1)-\frac{\pa u(y,x_s^1)}{\partial \nu(y)}u^{\rm inc}(y,x_s^2)\right]{\rm d}s(y),
\en
where $B_{\varepsilon}(x_s^j)$, $j=1,2,$ are two balls centered at $x_s^j$ with radius $\varepsilon>0$ which is small enough such that $B_{\varepsilon}(x_s^1)\cap B_{\varepsilon}(x_s^2)=\emptyset$. 
A direct application of the Green's theorem shows that 
\be\label{e3}
\lim_{\varepsilon\rightarrow 0}E_\varepsilon(x_s^1)=\lim_{\varepsilon\rightarrow 0}F_\varepsilon(x_s^2)\qquad {\rm for} \;\; x_s^1,x_s^2\in \R^2\se(\ov{D}\cup\G)
\en
with $x_s^1\not=x_s^2$, if $u(\cdot,x_s^j) \in L^p(B_\varepsilon(x_s^j))$, $p\geq1$.

Throughout, we consider two classes of incident wave $\Phi_{\kappa_1}(x,x_s)$ and $\partial_{x_\ell}\Phi_{\kappa_1}(x,x_s)$ for $\ell=1,2$. To distinguish these two incident waves, from now on, we will always denote them by $u^{\rm inc}(x,x_s):=\Phi_{\kappa_1}(x,x_s)$ and $\widetilde{u}^{\rm inc}(x,x_s,\ell):=\partial_{x_\ell} \Phi_{\kappa_1}(x,x_s)$. Moreover, to distinguish corresponding total waves and scattered waves, we will always use $u(x,x_s)$, $u^{s}(x,x_s)$ and $\widetilde{u}(x,x_s,\ell)$, $\widetilde{u}^{s}(x,x_s,\ell)$ to denote the total fields and the scattered fields associated with (\ref{a1}) corresponding to the incident wave $u^{\rm inc}(x,x_s)$ and $\widetilde{u}^{\rm inc}(x,x_s,\ell)$, respectively.

In (\ref{e1}) and (\ref{e2}),  we choose the incident field $u^{\rm inc}$ in the following form
\ben
u^{\rm inc}(x,x_s^1): = \Phi_{\kappa_1}(x,x_s^1),\qquad u^{\rm inc}(x,x_s^2)=\widetilde{u}^{\rm inc}(x,x_s^2,\ell):=\partial_{x_\ell} \Phi_{\kappa_1}(x,x_s^2).
\enn
Notice $\widetilde{u}(\cdot,x_s^2, \ell)$ is smooth in $B_\varepsilon(x_s^1)$ for sufficiently small $\varepsilon>0$ due to $x_s^1\not=x_s^2$. It is well-known by a standard argument that 
$\lim_{\varepsilon\rightarrow 0}E_\varepsilon(x_s^1) = \widetilde{u}(x_s^1,x_s^2, \ell)$.
One then has by (\ref{e3}) 
\be\label{e4}
\widetilde{u}(x_s^1,x_s^2, \ell) 
= \lim_{\varepsilon\rightarrow 0}\int_{\pa B_\varepsilon(x_s^2)}\left[\frac{\pa\widetilde{ u}^{\rm inc}(y,x_s^2, \ell)}{\partial \nu(y)}u(y, x_s^1)-\frac{\pa u(y,x_s^1)}{\partial \nu(y)}\widetilde{u}^{\rm inc}(y,x_s^2,\ell)\right]{\rm d}s(y),
\en
which will play an important role in the proof of the uniqueness of the inverse problem.

Our second goal is to prove uniqueness of $\Gamma$.
Assuming on the contrary $\Gamma_1\neq\Gamma_2$, without loss of generality, there exists $z^*\in\G_1\se\G_2$. Furthermore, we can choose $\varepsilon_0>0$ and $\delta_0>0$ such that   
\be\label{d1}
z_n:=z^*+\frac{\delta_0}{n}\nu(z^*)\in B_{\varepsilon_0}(z^*)\qquad{\rm for}\;\; n=1,2,\cdots
\en
and $B_{\varepsilon_0}(z^*)\cap\G_2=\emptyset$. 

Let $\widetilde{u}_{1}(\cdot,z_n,\ell)$, $\widetilde{u}_{2}(\cdot,z_n,\ell)$ denote the solutions to Problem (\ref{a1}) with $(\G, \kappa_2, D,\mathcal{B})$ replaced by  $(\G_1, k_{2,1}, D_1,\mathcal{B}_1)$,  $(\G_2, k_{2,2}, D_2,\mathcal{B}_2)$, respectively.
Then we aim to show that
\be\label{d2}
\widetilde{u}_{1}(y,z_n,\ell)=\widetilde{u}_{2}(y,z_n,\ell)\quad {\rm for} \quad y\in\Omega_1(\G_1)\cap\Omega_1(\G_2)\se\{z_n\}
\en
holds for $\ell=1,2$ and $n\in\N$, where $\Omega_1(\Gamma_j)$ and $\Omega_2(\Gamma_j)$ for $j=1,2$ denote the upper and lower half-spaces separated by $\Gamma_j$, respectively. In view of (\ref{e4}), we have 
\be\label{d3}
\widetilde{u}_m(y,z_n,\ell)=\lim_{\varepsilon\rightarrow 0}\int_{\pa B_\varepsilon(z_n)}\left[\frac{\pa \widetilde{u}^{\rm inc}(x,z_n,\ell)}{\partial \nu(x)}u_m(x, y)-\frac{\pa u_m(x,y)}{\partial \nu(x)}\widetilde{u}^{\rm inc}(x,z_n,\ell)\right]ds(x),
\en
for $m=1,2$. For $y\in\G_{b,a}$, it is apparent from $u_1^{s}(\cdot,y)=u_2^{s}(\cdot,y)$ on $\G_{b,a}$ and the analyticity of the scattered field that $u_1^{s}(\cdot,y)=u_2^{s}(\cdot,y)$ on $\G_b:=\{(x_1,x_2)\in\R^2:x_2=b\}$, which implies $u_1^{s}(\cdot,y)=u_2^{s}(\cdot,y)$ on $\Omega_1(\G_1)\cap\Omega_1(\G_2)$ due to the uniqueness of the Dirichlet problem and the analytic continuation principle. Hence, $u_1(\cdot,y)=u_2(\cdot,y)$ on $\Omega_1(\G_1)\cap\Omega_1(\G_2)\se\{y\}$. Thus, for any $z_n\in\Omega_1(\G_1)\cap\Omega_1(\G_2)$, we can choose sufficient small $\varepsilon>0$ such that $\pa B_{\varepsilon}(z_n)\subset\Omega_1(\G_1)\cap\Omega_1(\G_2)$, so we have $u_1(\cdot,y)=u_2(\cdot,y)$ on $\pa B_\varepsilon(z_n)$ which yields (\ref{d2}) holds for $y\in\G_{b,a}$ from (\ref{d3}). The analyticity, the uniqueness of the Dirichlet problem and the analytic continuation principle give that  (\ref{d2}) holds for $y\in\Omega_1(\G_1)\cap\Omega_1(\G_2)\se\{z_n\}.$

Define $v_1(\cdot,z_n,\ell):=\widetilde{u}_1(\cdot,z_n,\ell)|_{D_{z^*}}$,  $v_2(\cdot,z_n,\ell):=\widetilde{u}_2(\cdot,z_n,\ell)|_{D_{z^*}}$ where $D_{z^*}:=B_{\varepsilon_0}(z^*)\cap\Omega_1(\G_2)\cap\Omega_2(\G_1)$ which can be chosen as a Lipschitz domain with a sufficiently small $\varepsilon_0>0$.
According to the definitions, it is clear that $v_{1}$ and $v_{2}$ solve the following interior transmission problem 
\be\label{d5}
\left\{\begin{array}{lll}
         \Delta v_{1}+\kappa_{2,1}^2v_{1}=0 \qquad\textrm{in} \quad D_{z^*}, \\
         \Delta v_{2}+\kappa_1^2v_{2}=0 \qquad\;\;\textrm{in} \quad D_{z^*}, \\
         v_{2}-v_{1}=f^{\ell}_{1}\qquad\qquad \textrm{on} \quad \partial D_{z^*}, \\
         \partial_{\nu} (v_{2}-v_{1})=f^{\ell}_{2}\qquad\;\textrm{on} \quad \partial D_{z^*},
       \end{array}
\right.
\en
with $f^{\ell}_{1}\in H^{\frac{1}{2}}(\pa D_{z^*})$ and $f^{\ell}_{2}\in H^{-\frac{1}{2}}(\pa D_{z^*})$. 
Furthermore, it follows from (\ref{d2}) that $f^{\ell}_{1}=0$ and $f^{\ell}_{2}=0$ on $B_{\varepsilon_0}(z^*)\cap\G_1$. 
With the aid of this, we conclude from the well-posedness (cf.\cite{YZZ13}) of Problem (\ref{d5}) for sufficiently small $\varepsilon_0$ with the a priori estimates (\ref{c20}),(\ref{c26}), (\ref{c29}) in $L^p(1<p<2)$ for Problem (\ref{a1}) that the following estimate 
\be\label{d17}
\|\sum_{\ell=1,2}\nu_{\ell}(z^*)v_{1}(\cdot, z_n, \ell)\|_{L^2(D_{z^*})}+\|\sum_{\ell=1,2}\nu_{\ell}(z^*)v_{2}(\cdot,z_n, \ell)\|_{L^2(D_{z^*})}\leq C
\en
holds for uniformly $n\in\mathbb N$, where $\nu(z^*):=(\nu_1(z^*), \nu_2(z^*))$ denotes the upward unit normal at $z^*$. By linearity, it is clear that $\sum_{\ell=1,2}\nu_{\ell}(z^*)v_{2}(\cdot,z_n, \ell)$ is the solution to Problem (\ref{a1}) with $\G=\G_2$ associated with the incident wave $\sum_{\ell=1,2}\nu_{\ell}(z^*)\widetilde{u}^{\rm inc}(x,z_n,\ell)$.
Therefore, we have
\ben\no
\|\sum_{\ell=1,2}\nu_{\ell}(z^*)v_{2}(\cdot,z_n, \ell)\|_{L^2(D_{z^*})}
&\geq&\|\sum_{\ell=1,2}\nu_{\ell}(z^*)\widetilde{u}^{\rm inc}(\cdot,z_n,\ell)\|_{L^2(D_{z^*})}\\\label{d18}
&&-\|\sum_{\ell=1,2}\nu_{\ell}(z^*)\left[v_{2}(\cdot,z_n, \ell)-\widetilde{u}^{\rm inc}(\cdot,z_n,\ell)\right]\|_{L^2(D_{z^*})},
\enn
by the triangle inequality. A direct calculation leads to $\|\sum_{\ell=1,2}\nu_{\ell}(z^*)\widetilde{u}^{\rm inc}(\cdot,z_n,\ell)\|^2_{L^2(D_{z^*})}=O(n)$ for sufficiently large $n$. Since there is a positive distance between $z^*$ and $\Gamma_2$, it can be obtained by the well-posedness of Problem (\ref{a1}) with $\Gamma=\Gamma_2$ that the scattered field $v_{2}(\cdot,z_n, \ell)-\widetilde{u}^{\rm inc}(\cdot,z_n,\ell)$ is uniformly bounded for all $n\in\mathbb N$. Hence 
\ben
\|\sum_{\ell=1,2}\nu_{\ell}(z^*)v_{2}(\cdot,z_n, \ell)\|_{L^2(D_{z^*})}\rightarrow +\infty\quad{\rm as}\;\; n\to\infty
\enn
which contradicts with (\ref{d17}). So, $\G_1=\G_2$.

Now we are in position to show $\kappa_{2,1}=\kappa_{2,2}$. Let $\G:=\G_1=\G_2$ and assume $\kappa_{2,1}\neq \kappa_{2,2}$. We choose $z^*\in\G$ and define $z_n$ as (\ref{d1}). Recalling that $\widetilde{u}_{1}(\cdot,z_n,\ell)$, $\widetilde{u}_{2}(\cdot,z_n,\ell)$ denote the solutions of the scattering problems (\ref{a1}) with $(\G, \kappa_2, D,\mathcal{B})$ replaced by  $(\G, \kappa_{2,1}, D_1,\mathcal{B}_1)$,  $(\G, \kappa_{2,2}, D_2,\mathcal{B}_2)$, respectively,  with the point sources $\widetilde{u}^{\rm inc}(x,z_n,\ell)=\partial_{x_{\ell}}\Phi_{\kappa_1}(x,z_n)$. By the same analysis as (\ref{d2}), we conclude that
\be\label{d19}
 \widetilde{u}_{1}(x,z_n,\ell)=\widetilde{u}_{2}(x,z_n,\ell)\quad{\rm in} \quad\Omega_1.
 \en
 For $z^*\in\G$, we choose a Lipschitz domain $D'_{z^*}:=B_{\varepsilon_0}(z^*)\cap\Omega_2$ and define $\widetilde{v}_{1}(x,z_n,\ell):=\widetilde{u}_{1}(x,z_n,\ell)|_{D'_{z^*}}$ and  $\widetilde{v}_{2}(x,z_n,\ell):=\widetilde{u}_{2}(x,z_n,\ell)|_{D'_{z^*}}$. Then $\widetilde{v}_{1}$ and $\widetilde{v}_{2}$ satisfy the following interior transmission problem in a different domain $D'_{z^*}$
\ben\label{d20}
\left\{\begin{array}{lll}
         \Delta \widetilde{v}_{1}+\kappa_{2,1}^2\widetilde{v}_{1}=0 \qquad \textrm{in} \quad D'_{z^*}, \\
         \Delta \widetilde{v}_{2}+\kappa_{2,2}^2\widetilde{v}_{2}=0 \qquad \textrm{in} \quad D'_{z^*}, \\
         \widetilde{v}_{2}-\widetilde{v}_{1}=g^{\ell}_{1}\qquad\qquad \textrm{on} \quad \partial D'_{z^*}, \\
         \partial_{\nu} (\widetilde{v}_{2}-\widetilde{v}_{1})=g^{\ell}_{2}\qquad\;\textrm{on} \quad \partial D'_{z^*},
       \end{array}
\right.
\enn
with $g^{\ell}_{1}\in H^{\frac{1}{2}}(\pa D'_{z^*})$ and $g^{\ell}_{2}\in H^{-\frac{1}{2}}(\pa D'_{z^*})$. From (\ref{d19}), we deduce that $g^{\ell}_{1}=0$ and $g^{\ell}_{2}=0$ on $B_{\varepsilon_0}(z^*)\cap\G$. The remaining part of this proof is completely analogous to the proof of uniqueness for $\Gamma$, so we omit it here. 

Due to $\Gamma_1=\Gamma_2$ and $\kappa_{2,1}=\kappa_{2,2}$, the inverse problem is reduced to the case of determining the embedded obstacle from measurements in a known two-layered medium. It follows from the standard discussions (cf. \cite[Theorem 5.6]{CK13}) that $D_1=D_2$ and $\mathcal{B}_1=\mathcal{B}_2$. The proof is completed.
\end{proof}

\subsection{The case of an inhomogeneous medium $n(x)$}
In this subsection, we investigate the case where $D$ is a penetrable inhomogeneous medium. A global uniqueness is obtained which shows that the locally rough interface, the wave number in the lower half-space, and the inhomogeneous medium can be uniquely determined by the scattered field measured on $\Gamma_{b,a}$.

\begin{theorem}\label{thm8}
If $u_1^{s}(x,x_s)=u_2^{s}(x,x_s)$ for all $x,x_s\in\Gamma_{b,a}$, then $(\Gamma_1,\kappa_{2,1}, n_1)=(\Gamma_2,\kappa_{2,2}, n_2)$. 
\end{theorem}
\begin{proof}
It follows from the same arguments as Theorem \ref{thm7} that $\Gamma_1=\Gamma_2$ and $\kappa_{2,1}=\kappa_{2,2}$.
Thus, the inverse problem is reduced to recover the inhomogeneous medium into a known two-layered medium and the uniqueness $n_1=n_2$ directly follows from a standard discussion \cite{ALB08}.
\end{proof}

\end{document}